\documentclass{amsart}[12pt]
\usepackage{amssymb,latexsym,amscd,amsthm,amsfonts,enumerate,array,bbm,bm}
\usepackage{amsmath}
\usepackage{graphicx}
\usepackage{eufrak}

\usepackage{pb-diagram,lamsarrow,pb-lams}
\usepackage{hyperref}

\numberwithin{equation}{section}

\newtheorem{theorem}{Theorem}[section]
\newtheorem{thm}[theorem]{Theorem}
\newtheorem{proposition}[theorem]{Proposition}

\newtheorem{corollary}[theorem]{Corollary}
\newtheorem{lemma}[theorem]{Lemma}

\theoremstyle{definition}

\newtheorem{remark}[theorem]{Remark}
\newtheorem{example}[theorem]{Example}
\newtheorem{definition}[theorem]{Definition}

\newtheorem{prop}[theorem]{Proposition}

\newcommand{\divisor}{{\textsc{div}}}
\newcommand{\rightdivides}{|_{\rightarrow}}
\newcommand{\bigdiamond}{\scalebox{1.5} {$\diamond$}}
\newcommand{\arcs}{\bf{\rm{Arcs}}}

\newcommand{\dpair}[2]{[[#1,#2]]}

\newcommand{\degree}{\textrm{deg}}
\newcommand{\lcm}{\textrm{lcm}}

\newcommand{\DD}{\mathcal{D}}

\newcommand{\LL}{\mathcal{L}}
\newcommand{\MM}{\mathcal{M}}
\newcommand{\SSS}{\mathcal{S}}
\newcommand{\KK}{\mathcal{K}}

\newcommand{\YY}{\mathcal{Y}}
\newcommand{\ZZ}{\mathcal{Z}}
\newcommand{\mfK}{\mathfrak{K}}

\newcommand{\arcspan}[2]{\textrm{arcspan}_{#1}{#2}}

\begin{document}
\title[On the rational relationships among pseudo-roots]
{On the rational relationships among pseudo-roots of a non-commutative polynomial}

\author{Vladimir Retakh}
\address{\noindent Department of Mathematics, Rutgers University,
Piscataway, NJ 08854, USA} \email{vretakh@math.rutgers.edu}

\author{Michael Saks}
\thanks{M.S. is supported in part by Simons Foundation under grant 332622.}
\address{Department of Mathematics, Rutgers University, Piscataway, NJ 08854, USA}
\email{saks@math.rutgers.edu}


\begin{abstract} 
For a non-commutative ring $R$, we consider factorizations of polynomials in $R[t]$ where
$t$ is a central variable.  A pseudo-root of a polynomial $p(t)=p_0+p_1t+ \cdots p_kt^k$ is
an element $\xi \in R$, for which there exist polynomials $q_1,q_2$ such that $p=q_1(t-\xi)q_2$.
We investigate the rational relationships that hold among the pseudo-roots of $p(t)$ by using the diamond
operations for cover graphs of modular lattices.

When $R$ is a division ring, each finite subset $S$ of $R$ corresponds to a unique minimal monic polynomial $f_S$ that vanishes on $S$. 
By results of Leroy and Lam \cite{LL},
the set of polynomials $\{f_T:T \subseteq S\}$ with the right-divisibility order forms a lattice with join operation corresponding to
(left) least common multiple and meet operation corresponding to (right) greatest common divisor.  The set of edges of the cover
graph of this lattice correspond naturally to a set $\Lambda_S$ of pseudo-roots of $f_S$.
Given an arbitrary subset of $\Lambda_S$,  our results provide a graph theoretic criterion that guarantees that the
subset rationally generates all of $\Lambda_S$, and in particular, rationally generates $S$.
\end{abstract}

\maketitle

\bigskip
\noindent

\section{Introduction}
The theory of polynomials with noncommutative coefficients and central variables was initiated by Wedderburn, Dickson and Ore
(see, e.g., \cite{L},Chapter 5.16 and \cite{O}).  
There is a significant literature on
polynomials with matrix coefficients and their factorizations into linear factors, for example, \cite {B, GLR, Od}. 
For factorizations of noncommutative polynomials in a more
general setting see, for example \cite {LL, LL1}.

Let $R$ be an arbitrary ring, and
consider the ring $R[t]$ where $t$ commutes with all elements of $R$.   Since $t$ is central, every product
$a_1 \cdots a_r$ with $a_i \in R \cup \{t\}$ is equal to a monomial of the form $a t^d$, and every element
of $R[t]$ has a normal form $p=p_0+p_1t+ \cdots + p_kt^k$. 
As usual, for $\alpha \in R$, the evaluation $p(\alpha)$ is defined to be $p_0+p_1\alpha + \cdots + p_k\alpha^k$, and $\alpha$ is a {\em zero} of $p$ if $p(\alpha)=0$.
Some familiar properties of polynomials over commutative rings fail for non-commutative rings. 
The identity $pq(\alpha)=p(\alpha)q(\alpha)$ need not hold. While every zero of $q$ is a zero of $pq$, a zero of $p$ need not be a zero of $pq$.
A degree $d$ polynomial may have
more than $d$ distinct zeros, e.g., the polynomial $t^2+1$ over quaternions has infinitely many roots.

Following \cite{RSW, GGRW1}, an element $\xi \in R$ is  a {\em pseudo-root}
of $p$ provided there exist polynomials $q_1,q_2\in R[t]$ such that
$p=q_1(t-\xi)q_2$.   
We call $\xi$ a {\em right root} of $p$ if $q_2=1$, and a {\em left root} of
$p$ if $q_1=1$.  
It is easy to verify that $\xi$ is a zero of $p$ if and only if $\xi$ is a right-root of $p$.
Let $Z(p)$ be the set of zeros of $p$ and $\Lambda(p)$ be the set of pseudo-roots of $p$.

If the polynomial $p(t)$ factors as $(t-\alpha_1)\cdots (t-\alpha_d)$, then  $\alpha_1,\ldots,\alpha_d$ are pseudo-roots, and $\alpha_d$ is a zero.
For a commutative domain $R$, of course, every permutation
of the factors  is a factorization,  and $Z(p)=\{\alpha_1,\ldots,\alpha_d\}$.
In the non-commutative case, a
permutation of a factorization need not be a factorization, and a polynomial may have many factorizations that are not
equivalent under permutation. 

Throughout the paper, we assume that $R$ is a division ring.
Therefore
$R[t]$ is a left principal ideal ring and the  set  of monic polynomials is in 1-1 correspondence
with the left ideals.  Thus $R[t]$ with the divisibility order is a lattice $\mathcal{L}(R)$ with join operation $\lcm(p_1,p_2)$ (least common multiple)
equal to the unique
monic generator of the left ideal $R[t]p_1 \cap R[t]p_2$, and
meet operation $\gcd(p_1,p_2)$ (greatest common divisor) 
equal to the unique monic generator of $R[t]p_1+R[t]p_2$. \footnote{In the non-commutative setting there are
two different gcd operations and two different lcm operations, depending on whether one focuses on left or right ideals; in this paper our choice of gcd and lcm is determined by our focus on left ideals.} This lattice is necessarily {\em modular}  (see Section
~\ref{subsec:mod-distrib}.) 

Connections among pseudo-roots for polynomials over division rings is given by the famous Gordon-Motzkin theorem 
(see \cite{GM} or Chapter 5.16 in \cite{L}):  The zeros of any polynomial $p$ 
lie in at most $\degree(p)$
conjugacy classes of $R$ and if $p$ factors as $(t-\alpha_1)\dots (t-\alpha_d)$ 
then each zero of $p$ is conjugate to some $\alpha_i$.

Exact conjugation formulas connecting zeros and pseudo-roots over division rings were given by Gelfand and Retakh in \cite {GR2} (see also \cite {GR3, GGRW}).
They expressed coefficients of polynomial $p=t^n+a_1t^{n-1}+\dots+a_n$ as rational functions of zeros $\xi_1,\dots, \xi_n$ provided that the roots are in {\em generic position} 
(see Section ~\ref{subsec:example}).

The following simple example from \cite{GR2} is instructive.

\begin{example}
\label{example:quadratic}
Given elements $x_1,x_2 \in R$ that are "suitably generic", there are unique 
elements $x_1'$ and $x_2'$ such that 
$(t-x_1')(t-x_1)=(t-x_2')(t-x_2)$.  Call this polynomial $p(t)$.  Then $x_1,x_2,x_1',x_2'$ are all pseudo-roots of $p(t)$
with $x_1,x_2$ being zeros.   We can express $x_1'$ and $x_2'$ as rational functions of $x_1,x_2$ via the formulas
$$
x_1'=(x_1-x_2)x_1(x_1-x_2)^{-1}, \ x_2'=(x_2-x_1)x_2(x_2-x_1)^{-1},$$
and can express $x_1,x_2$ as rational functions of $x_1',x_2'$ via the formulas
$$
x_1=(x_1'-x_2')^{-1}x_1'(x_1'-x_2'), \  x_2=(x_2'-x_1')^{-1}x_2'(x_2'-x_1'),
$$
provided that the needed inverses are well-defined (this is what is meant by the above requirement that $x_1,x_2$ be "suitably generic").   However, 
one cannot  (in general) rationally express  either $x_2$ or $x_2'$ in terms of $x_1, x_1'$ \cite{Be}. 
\end{example}

This example suggests the following general problem:  given a set  $B$ of pseudo-roots of a polynomial $p$ and another
pseudo-root $\alpha$ is $\alpha$ rationally generated by $B$?  This is the focus of the present paper.

To formalize our problem, we need a way to specify individual pseudo-roots of a polynomial.   As we now describe,
there is a natural directed graph whose edges correspond to pseudo-roots of $p$.
For  monic polynomials $q,r$, $q$ is a {\em right divisor} of $r$, denoted $q\rightdivides r$, if there is a polynomial $s$
such that $r=sq$.  The polynomial $s$ is unique, 
and is denoted by $r/q$.  
Let $G=G(R)$ be the directed graph on vertex set  $R^M[t]$, the set of monic polynomials in $R[t]$,
 whose arc-set $A(R)$ consists
of pairs $ r\rightarrow q$ for all $q,r$ such that $q \rightdivides r$ and $\degree(q)=\degree(r)-1$.
Every arc $ r\rightarrow q$ can be naturally associated to an element $\psi(r \rightarrow q)=t-r/q$ of $R$.

Let $G_p=G_p(R)$ be the restriction of $G(R)$ to the set $\divisor(p)$ of right divisors of $p$, and let $A_p$ be the arc-set of $G_p$.  
In the case of present interest where $R$ is a division ring and $p$ is factorizable, $G_p$ is the {\em cover graph} of the lattice of divisors
of $p$, i.e., for any divisors $q$ and $r$ of $p$, $q \rightdivides r$ if and only if there is a directed path from $r$ to $q$ in $G_p$.  
It is easy to check from the definitions that
for every $r \rightarrow q \in A_p$, $\psi(r \rightarrow q)$ is a pseudo-root of $p$, and  every pseudo-root is representable in this form for some (not
necessarily unique) arc of $A_p$. 

The lattice $\divisor(p)$ is, in general, infinite and so is the set of pseudo-roots of $p$, and one gains greater control over the problem
by restricting to certain natural finite sublattices, that were studied by
Lam and Leroy ~\cite{LL}.  If $S \subseteq R$ is finite, then the set of polynomials that vanish at every $s \in S$ is a left ideal and so
is generated by a unique monic polynomial, which we denote by $f_S$.  Polynomials of the form $f_S$ were studied
by Lam and Leroy \cite{LL}, who called them {\em Wedderburn polynomials}.  In the case that $R$ is a field,
$f_S$ is just the product $\prod_{s \in S}(t-s)$.  As shown by Lam and Leroy (see \cite{LL})
for any subset $S$, the set $\{f_T:T \subseteq S\}$ is closed under $\gcd$ and $\lcm$\footnote{The closure
under $\lcm$ follows from the simple fact that $\lcm(f_S,f_W)=f_{S \cup W}$ but closure under
$\gcd$ is less obvious; it is not in general true that $\gcd(f_S,f_W)=f_{S \cap W}$.} and is thus a sublattice
of
$\mathcal{L}(R)$. We denote
this sublattice by $\mathcal{L}_S$. Note that $\mathcal{L}_S$ is a finite sublattice of the lattice $\divisor(f_S)$, and has
at most $2^{|S|}$
elements (polynomials).  The sublattices of the form $\mathcal{L}_S$ give a rich source of examples
of finite sublattices of $\mathcal{L}(R)$.

Let $G_S$ be the subgraph of $G_{f_S}$ consisting only of the vertices in $\mathcal{L}_S$ and the arcs between them and write
$A_S$ for the set of arcs of $G_S$.  Then, as above, each arc $r \rightarrow q$ corresponds to a pseudo-root  $\psi(r \rightarrow q)$
of $f_S$.  Let $\Lambda_S$ denote the set of pseudo-roots of $f_S$ corresponding to  the arc set $A_S$.   We refer to $\Lambda_S$
as {\em $S$-pseudo-roots}.  It is easy to see that if $\deg(f_S)=k$
then every directed path from the maximal element $f_S$ to the minimal element 1,
consists of a sequence $f_k=f_S,f_{k-1},\ldots,f_0=1$ of polynomials where $\deg(f_{j})=j$.
If $a_j$ is the pseudo-root associated to the arc $f_j \rightarrow f_{j-1}$ then the
product $(t-a_k)(t-a_{k-1})\cdots (t-a_1)$ is a factorization of $f_S$.  In this way every
path from $f_S$ to 1, corresponds to a factorization of $f_S$ into linear factors where each factor
is $t$ minus a pseudo-root.

For a subset $B$ of $R$ let $\Phi(B)$ denote the closure of $B$ under ring operations
and inversion of units. For a subset $S$ we want to consider the restriction of this closure to the set $\Lambda_S$
of $S$-pseudo-roots.  More precisely, for $B \subseteq \Lambda_S$, let $\Phi_S(B)=\Phi(B) \cap \Lambda_S$, i.e., the set of $S$-pseudo-roots that are rationally generated by $B$.
The map $\Phi_S:\mathcal{P}(\Lambda_S) \longrightarrow \mathcal{P}(\Lambda_S)$ is a {\em closure map on $\Lambda_S$} (see Section ~\ref{subsec:closure}), and our goal is to provide a partial description of this map.

Under the correspondence between arcs of $A_S$ and $S$-pseudo-roots $\Lambda_S$, the closure operator $\Phi_S$ on $\Lambda(p)$ 
can be interpreted as a closure operator on  the set $A_S$ of arcs of $G_S$: For 
$B \subseteq A_S$, $\Phi_S(B)=\{r \rightarrow q \in A_S:\psi(r \rightarrow q) \in \Phi(\psi(B))$, where $\psi(B)=\{\psi(r \rightarrow q):r \rightarrow q \in B\}$. 

While $\Phi_S$ depends on the ring-theoretic structure of $S$ within $R$, 
it was observed in ~\cite{GGRW1} that $\Phi_S$ can be partially captured by
a natural closure on the arc-set $A_S$, introduced in work of I. Gelfand and Retakh \cite{GR2, GR3},  that depends only on the graph structure of $G_S$.    This closure, called the {\em diamond-closure}, and denoted by $\bigdiamond$ can
be described briefly as follows (see Section ~\ref{subsec:diamond-closure} for a more precise definition).  If $\{a,b,c,d\}$ are vertices in $G_S$ such that $a \rightarrow b, a \rightarrow c$ and $b \rightarrow d$ and $c \rightarrow d$, then the set $\{a \rightarrow b,a \rightarrow c,b \rightarrow d,c \rightarrow d\}$ is a {\em diamond} of $G_S$.
The diamond closure of  $B \subseteq A_S$, denoted $\bigdiamond(B)$ is the smallest set of arcs containing $B$ that satisfies that
for any diamond $\{a \rightarrow b,a \rightarrow c,b \rightarrow d,c \rightarrow d\}$ if $\bigdiamond(B)$ contains
$\{a \rightarrow b,a\rightarrow c\}$ or $\{b \rightarrow d,c \rightarrow d\}$ then $\bigdiamond(B)$ contains the entire diamond.
This definition is motivated by Example ~\ref{example:quadratic}, where $a$ is the polynomial $p(t)$, $b$ is $t-x_1$ , $c$ is $t-x_2$
and $d$ is 1, and arc $a \rightarrow b$ corresponds to pseudo-root $x_1'$, $a \rightarrow c$ corresponds to pseudo-root $x_2'$,
$b \rightarrow d$ corresponds to $x_1$ and $c \rightarrow d$ corresponds to $x_2$.  The analysis of Example ~\ref{example:quadratic}
can be extended (see Proposition ~\ref{prop:weakening}) to show that
for any ring $R$
 and finite $B \subseteq R$, 
every member of $\bigdiamond(B)$ is an $S$-pseudo-root that is rationally generated
by $B$.

In this paper we investigate the $\bigdiamond$-closure for finite graphs $G$ that are cover graphs of modular lattices. This includes
the motivating situation that the graph $G$ is equal to 
$G_S$ for $S \subseteq R$. We give an explicit description of $\bigdiamond$-closed sets (Theorem ~\ref{thm:diamond-closed})
and use this to describe a simple procedure (see Theorem ~\ref{thm:MLDC}) for determining $\bigdiamond(B)$ for any subset $B$ of arcs.    
In particular, this provides a sufficient condition for a set of 
pseudo-roots in $\Lambda_S$ to rationally generate all of $\Lambda_S$. In the special case that the lattice is distributive, we can use the
Birkhoff representation for distributive lattices to provide a more explicit  characterization of connected diamond-closed
sets (Theorem ~\ref{thm:DL connected DC}) which yields a  simpler sufficient condition for a set of pseudo-roots of $\Lambda_S$
to rationally generate $\Lambda_S$.

Our methods can be also applied to more general situations including skew-polynomials (see \cite{LL}), linear differential operators \cite{GRW}  and 
principal ideal rings \cite{CDK}.

\section{Preliminaries}
\subsection{Closure spaces}
\label{subsec:closure}

A {\em closure space} is a pair $(X,\lambda)$ consisting of a set $X$ and a {\em closure map}
$\lambda:2^X \longrightarrow 2^X$ that satisfies (1) $A \subseteq \lambda(A)$,
(2) $A \subseteq B$ implies $\lambda(A) \subseteq \lambda(B)$ and (3) $\lambda(\lambda(A))=\lambda(A)$.
The image $\{A:\lambda(A)=A\}$ of $\lambda$ is denoted $\mathcal{C}_{\lambda}$.
Members of $\mathcal{C}_{\lambda}$ are  {\em $\lambda$-closed subsets}.
If $C$ is $\lambda$-closed and $A \subseteq C$ satisfies $\lambda(A)=C$
we say that $A$ is a {\em $\lambda$-generating set for $C$}.

If $\lambda$ and $\mu$ are closure operators on the same set $X$, we say that $\lambda$ is {\em weaker than} $\mu$
provided that $\lambda(A) \subseteq \mu(A)$ for all $A \in X$.  This is equivalent to the condition that every
$\mu$-closed set is also $\lambda$-closed.

An {\em alignment} on the set $X$ is a collection $\mathcal{C}$ of subsets of $X$ that
includes $X$ and is closed under arbitrary intersection.  
It is well-known and easy to show that there is a one-to-one correspondence between
the alignments on $X$ and closure spaces on $X$ as follows:  Given
a closure $\lambda$ on $X$, the set of $\lambda$-closed sets is an alignment, and 
an alignment ${\mathcal{C}}$ induces a closure map
$\lambda_{\mathcal{C}}(A)= \cap_{C \in \mathcal{C}:A \subseteq C} C$ for which $\mathcal{C}$
is the set of closed sets.

\subsection{Finite directed graphs}
\label{subsec:digraphs}

For a set $X$, an {\em arc} over $X$ is an ordered pair $(x,y)$ with $x \neq y$.
We denote an arc by $x \rightarrow y$ and say that
$x$ is the {\em tail} of the arc and $y$ the {\em head}.    We write $\arcs(X)$ for the set of all arcs over $X$. A directed graph (digraph) on $X$ is a set $A$ of arcs over $X$. 

For  $A \subseteq \arcs(X)$, we write $x \rightarrow_A y$ if $x \rightarrow y \in A$
and $x \sim_A y$ if $x \rightarrow_A y$ or $y \rightarrow_A x$.  We may omit
the subscript if $A$ is clear from context.
A vertex $x$ is {\em isolated in $A$} if
it belongs to no arc of $A$. The {\em vertex set spanned by $A$}, denoted $V(A)$, is the set
of vertices that are not isolated in $A$.  If $V(A)=X$ we say that {\em $A$ spans $X$}.
For $Y \subseteq X$, the set $A[Y]=A \cap \arcs(Y)$
is the {\em  set of arcs induced on $Y$ by $A$}. 
If $\mathcal{Y}$ is a collection of
subsets of vertices of $X$ we write $A[\mathcal{Y}]$ for the  arc-set
$\bigcup_{Y \in \mathcal{Y}}A[Y]$.

 A {\em path} in $A$ is a 
sequence $x_0,\ldots,x_k$ of vertices such that for each $i \in [k]$,
$x_{i-1} \sim_A x_i$.  We say that the path joins $x_0$ to $x_k$.
The path is {\em directed} or
a {\em dipath}
if $x_{i-1} \rightarrow_A x_i$ for each $i$. For $x,y \in X$, $x \asymp_A y$ means that
there is a path from $x$ to $y$ in $A$, and $x \rightsquigarrow_A y$ means that there
is a directed path from $x$ to $y$ in $A$.   A set $A$ of arcs is said to be {\em connected}
if for all $x,y \in V(A)$, $x \asymp_A y$. Note that a connected set of arcs need not
span the  entire vertex set. 
The maximal connected subsets $A_1,\ldots,A_K$ of $A$ are called the {\em arc-components}
of $A$.  The arc-components partition $A$.  The sets $V(A_i)$ partition $V(A)$ and are called the
{\em vertex-components} of $A$.

For $B \subseteq A$ we define $\arcspan{A}{B}$ to be the set of arcs of $A$ that join two vertices belonging to the
same vertex-component of $B$,  Thus $\arcspan{A}{B}$ is the maximal subset of $A$ whose vertex-components are the same
as the vertex-components of $B$. 

 \iffalse
We adapt the notation above, e.g. writing $x \rightarrow_G y$ to mean $x \rightarrow_{A(G)} y$
and $x \rightsquigarrow_G y$ to mean $x \rightsquigarrow_{A(G)} y$.
\fi

\subsection{Diamond-closure}
\label{subsec:diamond-closure}

We now introduce some non-standard notation particular to this paper.
A pair of arcs $x \rightarrow y$ and $x \rightarrow z$ having the same tail is called an {\em out-V} and is denoted $x \rightarrow \dpair{y}{z}$
and a pair of arcs $y \rightarrow w$ and $z \rightarrow w$ have the same head is called an {\em in-V} and is denoted
 $\dpair{y}{z} \rightarrow w$.  
If $x \rightarrow \dpair{y}{z}$ is an out-V, and $\dpair{y}{z} \rightarrow w$ is an in-V, then their union is denoted $x \rightarrow \dpair{y}{z} \rightarrow w$ and is called
a {\em diamond}.  For $A \subseteq \arcs(X)$, $\Delta(A)$ denotes the set of all diamonds in $A$.
A subset $S \subseteq \arcs(X)$ of arcs {\em spans} a diamond $D$ if $S$ contains the in-V of $D$ or the out-V of $D$, and we
say that $D$ is {\em spanned by $S$}.

Fix a subset $A$ of $\arcs(X)$. 
A subset $S \subseteq A$  is said to be {\em diamond-closed} with respect to $A$
if every diamond  $D \in \Delta(A)$  that is spanned by $S$ is a subset of $S$.
The intersection of diamond-closed subsets is diamond-closed, so is an alignment on the
set $\arcs(X)$, 
and has an associated closure map
 $\bigdiamond_A:2^{\arcs(X)} \longrightarrow 2^{\arcs(X)}$, where
$\bigdiamond_A(B)$ is the intersection of all diamond-closed subsets that contain $B$.
In most situations, the set $A$ is understood from context and we write simply $\bigdiamond$ for $\bigdiamond_A$.
We say that $B$ is a {\em $\bigdiamond$-generating set for set $C$} if $C \subseteq \bigdiamond(B)$.

\subsection{Lattices, Modular lattices and Modular lattice diagrams}
\label{subsec:lattices}

Let $L$ be a complete  lattice with operations
 $\vee$  (join) and $\wedge$ (meet) and associated
partial  order $\leq$ defined by $x \leq y$
 provided $x \vee y = y$, or equivalently $x \wedge y=x$. 
A complete lattice  necessarily has a unique minimum element,
$\hat{0}=\hat{0}_L=\wedge_{x \in L} x$ and a unique maximum element
$\hat{1}=\hat{1}_L=\vee_{x \in L} x$. 

We restrict attention to lattices of {\em finite length}, i.e.,  those for which there is an upper bound
on the length of the longest chain (totally ordered set).  

Since $L$ is complete,
the join and meet of any subset $A$ of elements is well-defined, and are denoted, respectively by  $\bigvee A$ and $\bigwedge A$. For $A =\emptyset$, $\bigvee A = \hat{0}$ and $\bigwedge A = \hat{1}$.

We say {\em $y$ covers $x$ in $L$}  (equivalently {\em $x$ is covered by $y$}), 
provided that $x <y$ and for all $z$ such that $x \leq z \leq y$ we have
$z=x$ or $z=y$.  
The cover set $C_L$ is the set of arcs 
$y \rightarrow_L x$ with $x,y \in L$ such that $y$ covers $x$. 
Since $L$ has finite length, $x < y$ if and only if there is a directed path from $y$ to $x$ in
the cover graph.

The {\em height function} of  lattice $L$ of finite height is the function $h=h_L:L \longrightarrow \mathbb{N}$,  given by
$h_L(x)$ is equal to the length (number of arcs) of the longest directed
path from $x$ to $\hat{0}$ in $C_L$. Thus $h_L(\hat{0})=0$.  The
{\em height} of $L$, $h_L$  is defined to be $h_L(\hat{1})$.  

The cover set of a lattice of finite height  is   connected, since there is a directed path from
any element to $\hat{0}_L$.
When
$y \rightarrow_L x$ we have $h(y)\leq h(x)+1$.  We say that
$L$ is {\em ranked} provided that $h(y)=h(x)+1$ whenever $y$ covers $x$.
We note the following well-known fact.  

\begin{prop}
\label{prop:ranked}  For a height-bounded lattice $L$, $L$ is ranked if and only if
for any $y,x$ with $y>x$, every directed path from $y$ to $x$
in the cover graph has the same length.
\end{prop}

A sublattice $K$ of $L$ is a subset closed under $\vee$ and $\wedge$.    Every subset of
$L$ with one element is a sublattice, and we say that a sublattice is {\em non-trivial} if
it contains at least two elements of $L$.  As a finite lattice, $K$ has a minimum $\hat{0}_K$
and $\hat{1}_K$ which need not be the same as $\hat{0}_L$ and $\hat{1}_L$.
The intersection of sublattices is a sublattice, so the set of sublattices is an alignment on $L$.
For $Y \subseteq L$,
the intersection of all sublattices containing $Y$ is denoted $\langle Y \rangle_L$ 
and is called the {\em sublattice  generated by $Y$}.

If $x,y \in K$ and $x$ covers $y$ in $L$ then clearly $x$ covers $y$ in $K$, but the converse
does not hold.  For example, one can have a sublattice $K$ for which $C_L[K]$ is empty (i.e,
$K$ contains no pair of elements that comprise a cover in $L$.)  Thus
the cover set $C_K$  contains the induced set of arcs $C_L[K]$, and the containment may be proper.
We say that $K$ is {\em cover-preserving} if $C_K=C_L[K]$.

We say that $K$ is a {\em connected sublattice} if $C_L[K]$ is connected.  A cover-preserving sublattice is necessarily connected (since $C_K$ is connected)
but a connected sublattice need not be cover-preserving.  We have the following
simple but useful criterion
for a lattice to be cover-preserving:

\begin{prop}
\label{prop:cover preserving}  A sublattice $K$ of the finite lattice $L$,
 is cover-preserving if and only if
for all $x,y \in K$ with $y \geq x$ there is a directed path 
$y=y_0,\ldots,y_k=x$ in $C_L$ with $y_0,\ldots,y_k \in K$.
\end{prop}

\begin{proof}
Assume $K$ is cover-preserving. Let $x,y \in K$ with $y>x$.  
Then there is a directed path
$y=y_0,\ldots,y_k=x$ in $C_K$, and since $C_K=C_L[K]$, 
this path also lies in $C_L$.

Conversely, assume that for all $x,y \in K$ with $y \geq x$ there is a directed path in $C_L$ from $y$ to $x$ with
all elements in $K$. Suppose $w$ covers $v$ in $K$; we need
that $w$ covers $v$ in $L$. Since $w>v$, by assumption there is a directed path 
 $w=w_0,\ldots,w_k=v$ in $C_L$ with all elements in $K$.
 If $k>1$, then $w>w_1>v$ contradicts that $w$ covers $v$ in $K$,
so $k=1$ and  $w$ covers $v$ in $C_L$.
\end{proof}

Next we consider the diamonds in $C_L$.  We say that $x,y$ are {\em $\vee$-neighbors}, denoted $x \sim_{\vee} y$, 
provided that $x \vee y$ covers both $x$ and $y$.  It is easy to check that $x \sim_{\vee} y$ if and only
if the cover graph has  an out-V with heads $x$,$y$, and this out-V is $x \vee y \rightarrow \dpair{x}{y}$.

Similarly, we say that $x,y$ are {\em $\wedge$-neighbors} denoted $x \sim_{\wedge} v$ provided
that $x$ and $y$ cover $x \wedge y$.  Then $x \sim_{\wedge} y$ if and only if the cover graph
has an in-V with heads $x$, $y$ and this in-V is $\dpair{x}{y} \rightarrow x \wedge y$.

If both $x \sim_{\vee} y$ and $x \sim_{\wedge} y$ then 
$x \vee y \rightarrow \dpair{x}{y} \rightarrow x \wedge y$.
In this case we say that $x,y$ are {\em diamond -neighbors} denoted
$x \sim_{\diamond} y$, and let $D(x,y)$ denote this diamond.  This is the only possible diamond
with $x$ and $y$ as the middle vertices.

\subsection{Modular and Distributive Lattices}
\label{subsec:mod-distrib}

Recall that a lattice is modular if for 
all $x,y,w \in L$ with $x \leq y$ we have $(x \vee w) \wedge y = x \vee (w \wedge y)$.
Modular lattices are so named because for any (left)-module over a ring, the lattice of submodules is  modular with $M \vee N= M+N$
and $M \wedge N = M \cap N$.
In particular 
the lattice of left-ideals of $R[t]$ is modular.

The following standard theorem (see  Theorem 2 and Corollary 3 of  Section 4.2 of \cite{Gra}) provides an alternative characterization of modularity for lattices of finite height.

\begin{thm}
\label{thm:modular}
Let $L$ be a lattice of finite height.  The following conditions are equivalent:
\begin{enumerate}
\item $L$ is modular.
\item The relations $\sim_{\vee}$, $\sim_{\wedge}$ and $\sim_{\diamond}$ are the same.
\item $L$ is ranked and the height function satisfies $h(x)+h(y)=h(x \wedge y)+h(x \vee y)$
for all $x,y$.
\end{enumerate}
\end{thm}

The second characterization is of special interest for us.

\iffalse
We note:
\begin{prop}
\label{prop:rank difference}
In a modular lattice, if $z,y,x$ are elements such that $z \geq y$
then $h(z \vee x)-h(z) \leq h(y \vee x)-h(y)$.  In particular, if $y \vee x$ covers 
either $z \vee x=z$ or $z \vee x$ covers $z$.
\end{prop}

This is standard; by the third characterization and the fact that $z \geq y$ we have:
$h(z \vee x)-h(z)=h(x)-h(z \wedge x) \geq h(x)-h(y \wedge x) = h(y \vee x) -h(y)$.
\fi

A lattice is {\em distributive} if $\wedge$ distributes over $\vee$ (and, equivalently,
$\vee$ distributes over $\wedge$).   In particular, distributive lattices are modular.
Any sublattice of the lattice of all subsets of some set (with
operations $\cup$ and $\cap$)
is distributive (and every distributive lattice is isomorphic to such a lattice). 

\subsection{Monic polynomials over division rings, diamond closure and pseudo-roots}
\label{subsec:polys}

The reader should review the definitions pertaining to the ring $R[t]$ from the introduction.
For a division ring $R$, $R[t]$ 
is a  left principal ideal ring.  There is a 1-1 correspondence between the set 
$R^M[t]$ of monic polynomials and the set of left ideals of $R[t]$ given by $p \leftrightarrow R[t]p$, where the right divisibility order of polynomials corresponds to
the (reverse) containment order on left ideals.   Since the left ideals of $R[t]$ form a modular lattice, the set $R^M[t]$ is
also a modular lattice, with $p \wedge q$ equal to the 
unique monic greatest common divisor $\gcd(p,q)$ (the
monic generator of the ideal $R[t]p+R[t]q$), and $p \vee q$ equal to the  least common multiple $\lcm(p,q)$,
which is the monic generator of $R[t]p \cap R[t]q$.  

As noted in the introduction, it was shown in \cite{LL} that for every finite subset $S$ of $R$ there
is a unique monic polynomial $f_S$ such that the ideal $R[t]f_S$ is the intersection of the ideals  $R[t](x-s)$ for $s \in S$.
The set $\{f_T:T \subseteq S\}$ ordered by right divisibility is a sublattice, denoted $\mathcal{L}_S$ of the lattice $\mathcal{L}(R)$ of polynomials
in $R[t]$ under right divisibility with $f \wedge g$ equal to the greatest common right divisor and $f \vee g$ equal to the
least common left multiple.    The graph $G_S$ was defined to be the cover graph of the lattice $\mathcal{L}_S$,
and its arc set is denoted $A_S$.

The following observation (essentially from ~\cite{GGRW1}) can be interpreted
as saying that
the rational  closure on $\Lambda_S$ is at least as strong as the $\bigdiamond$-closure on  $A_S$:

\begin{proposition}
\label{prop:weakening}
Let $R$ be a division ring and  $S \subseteq R $ be finite. Every $\Phi_S$-closed subset of $A_S$ is
$\bigdiamond$-closed. Thus $\bigdiamond$ is a weaker closure than $\Phi_S$.
In particular,  if $C \subseteq A_S$ and  $\bigdiamond(C)=A_S$ then $\Phi_p(C)=A_S$ where $p=f_S$.
\end{proposition}

\begin{proof}
Let $B$ be $\Phi_S$ closed.  Suppose that $q,r,s,u$ are right divisors of $f_S$ such that  $q \rightarrow \dpair{r}{s} \rightarrow u$ is a diamond and
$q \rightarrow \dpair{r}{s} \subset B$. 
We must show $\dpair{r}{s} \rightarrow u \subset B$.  (A similar proof
show that $\dpair{r}{s} \rightarrow u \subset B$ implies $q \rightarrow \dpair{r}{s} \subset B$.)

Let $a=\psi(q \rightarrow r)$, $b=\psi(q \rightarrow s)$, $c=\psi(r \rightarrow u)$ and $d=\psi(s \rightarrow u)$.
We will show that $c$ and $d$ can be expressed as rational functions of $a$ and $b$.  We have that $u=(t-c)(t-a)q$ and also
$u=(t-d)(t-b)q$ and so (since $q$ is not a 0-divisor)  $(t-c)(t-a) =(t-d)(t-b)$.  This implies that $c+a=d+b$ and $ca=db$, and
so $ca=(c+a-b)b$.  Rewriting this gives $c(a-b)=(a-b)b$.  Since $a \neq b$ (because $r,s$ are distinct polynomials) and
$R$ is a division ring we have $c=(a-b)b(a-b)^{-1}$.  Since $B$ is $\Phi_S$-closed and $a,b \in B$, we have $c \in B$.  Similarly $d \in B$.
\end{proof}

\subsection{An example: the diamond closure of a generic subset $X$ includes all $X$-pseudo-roots}
\label{subsec:example}

Gelfand and Retakh \cite{GR} showed that if $X$ is a suitably generic set of elements of the division ring $R$, then
$X$ rationally generates the entire set $\Lambda_X$ of $X$-pseudo-roots.   
Here the genericity requirement on 
$X=\{x_1,x_2,\dots, x_n\}$ is that any Vandermonde submatrix 
$$
V(i_1,i_2,\dots,i_k)=\left (\begin{matrix}
1&1&\dots&1\\
&\dots&\dots& \\
x_{i_1}^{k-2}&x_{i_2}^{k-2}&\dots&x_{i_k}^{k-2}\\
x_{i_1}^{k-1}&x_{i_2}^{k-1}&\dots&x_{i_k}^{k-1}
\end{matrix} \right )
$$
is invertible for pairwise distinct indices $i_1,i_2,\dots,i_k$ and the  inverse matrix has all entries non-zero.

Under this assumption, all of the Wedderburn polynomials $f_T$ for $T \subseteq X$ are distinct, and $\deg(f_T)=|T|$.
The sublattice $\{f_T:T \subseteq X\}$ is thus isomorphic to the power set of $X$ ordered by inclusion.  The  arcs of the
cover graph are of the form $T\cup\{u\} \longrightarrow T$ for $u \in X$ and $T \subseteq X-\{u\}$
and the pseudo-root associated to this arc is the unique element $x_{T,u}$ that satisfies $t-x_{T,u}=f_{T \cup \{u\}}/f_T$.
Each $u \in X$ corresponds to the arc $\{u\} \rightarrow \emptyset$ and the pair $(\emptyset,\{u\})$.

Example ~\ref{example:quadratic} corresponds to the special case that $|X|=2$.

It is easy to show that $\bigdiamond(X)$ includes all arcs $T \cup \{u\} \rightarrow T$ ( by induction on $T$) and
therefore $\bigdiamond(X)=A_X$.  By Proposition ~\ref{prop:weakening} this implies that all
$X$-pseudo-roots are rationally generated by $X$.

\section{Diamond-closure in a finite modular lattice}
\label{sec:DCS in ML}

In this section we provide a simple description of the diamond-closed sets of  a modular lattice of finite length.
Throughout this section $L$ denotes an arbitrary finite modular lattice and $C_L$ is its cover set.

We note the following easy observation:
\begin{prop} 
\label{prop:closed decomp}
A set $S \subseteq C_L$ is diamond-closed if and only if each
arc-component of $S$ is diamond-closed.
\end{prop}

The set $\bigdiamond(S)$ can be constructed from $S$ by the following direct procedure: If there is any diamond $D$ spanned by $S$ that is not contained in $S$ then replace
$S$ by $S \cup D$.   The definition of $\bigdiamond(S)$, implies that this operation preserves diamond-closure, 
and therefore  when the process terminates (which it must
by finiteness), we have $\bigdiamond(S)$. In particular,  $S$  $\bigdiamond$-generates $A$ if  this procedure produces all of $A$. 

This direct construction
provides little
insight into the structure of sets that $\bigdiamond$-generate $A$.
We will give a more explicit characterization of $\bigdiamond(S)$ for graphs
arising from modular lattices
 that is (1) easier to check for a given set $S$,
 and (2)  provides more insight into the structure
of sets that $\bigdiamond$-generate $A$.
Our characterization is particularly simple when $L$ is distributive.

\subsection{Diamond-closed sets of finite modular lattices}

 We begin
by observing that every sublattice of a finite lattice $L$ is diamond closed with respect to the graph $C_L$:

\begin{prop}
\label{prop:sublattice}  Let $L$ be a finite lattice and $K$ a sublattice.  The set of arcs
$C_L[K]$ is diamond-closed.
 \end{prop}

\begin{proof}
 Let $D= x \vee y \rightarrow \dpair{x}{y} \rightarrow
x \wedge y$ be a diamond of $C_L$ and suppose that $C_L[K]$ contains
$x \vee y \rightarrow \dpair{x}{y}$ or $\dpair{x}{y} \rightarrow x \wedge y$.  Then
$x,y \in K$ and since $S$ is a sublattice, $x  \wedge y$ and $x \vee y$ are
both in $K$, and so $D \subseteq C_L[K]$).  
\end{proof}  

More generally, a {\em sublattice packing} of $L$
is a set $\mfK=\{K_1,\ldots,K_s\}$ of disjoint sublattices.
Define $C_{{\mfK}}=\bigcup_{i} C_L[K_i]$, i.e., the union
of the arc-sets induced by the $C_L$ on each $K_i$. 
(This is a subset of, but not necessarily the same as, the arc-set 
induced by $C_L$ on $\bigcup_i K_i$
since it doesn't include arcs between different  $K_i$.)  Then
$C_{{\mfK}}$ is diamond-closed, 
(since the union of  diamond-closed subsets of arcs that are pairwise disconnected
is diamond-closed).

The main result of this section is that for modular lattices these are the only diamond-closed subsets, and further we
can restrict the $K_i$ to be cover-preserving sublattices of $L$. (This is not true for non-modular lattices.)
We say a sublattice packing $\mfK$ is a CS-packing if all of its members
are cover-preserving sublattices.

\begin{thm}
\label{thm:diamond-closed}
Let $L$ be a finite modular lattice. For $S \subseteq C_L$ (i.e., a set of arcs in the cover
graph) the following are equivalent:

\begin{enumerate}
\item $S$ is diamond-closed. 
\item Each arc-component of $S$ is induced from $C_L$ by a cover-preserving sublattice.
\item $S=C_L[{\mfK}]$ for some CS-packing $\mfK$.
\end{enumerate}
\end{thm}

\begin{proof}

The implication $(2) \Rightarrow (3)$ follows from the definitions of $C_L[\mfK]$  CS-packing, and the fact that
every cover-preserving sublattice is connected.
We have already noted the implication $(3) \rightarrow (1)$.  It remains to prove 
$(1) \Rightarrow (2)$.   For this it is enough to consider the case that $S$ is connected
and show that there is a cover-preserving sublattice $K$ such that $S=C_L[K]$.
The general case follows by applying this result to each connected arc-component of $S$.

\iffalse
\begin{thm}
\label{thm:connected diamond-closed}
Let $L$ be a finite modular lattice and let $S$ be a connected arc subset of $C_L$  Then 
$S$ is diamond-closed if and only if $S$ is equal to $C_L[K]$ for some cover-preserving sublattice $K$.
\begin{enumerate}
\item $S$ is diamond-closed,
\item There is a cover preserving sublattice $K$ such that $S$ is equal to $C_L[K]$. 
\item There is a connected sublattice $K$ such that $S$ is equal to $C_L[K]$.
\end{enumerate}
\end{thm}
\fi

We'll need some additional facts about finite modular lattices. 
Some of these facts are well-known, we include their easy proofs for completeness.
Throughout this proof, $L$ denotes an arbitrary modular lattice of finite length with height function $h$ and $x$ and $y$ are arbitrary
elements of $L$.

\begin{prop}
\label{prop:vee}
Suppose $x' \in L$ and $x'$ covers $x$.
Then either $x' \vee y = x \vee y$ and $x' \wedge y$ covers $x \wedge y$
or $x' \vee y$ covers $x \vee y$ and $x' \wedge y = x \wedge y$.  
\end{prop}

\begin{proof}
 We have $h(x' \vee y)+h(x' \wedge y)=h(x')+h(y)=1+h(x)+h(y)=1+h(x \vee y)+h(x \wedge  y)$.
Since $h(x' \vee y) \geq h(x \vee y)$ and $h(x' \wedge y) \geq h(x \wedge y)$, one
of the two alternatives stated in the proposition must hold.
\end{proof}

Define
$d(x,y)$ to be the length of the shortest (undirected) path in $C_L$ between
$x$ and $y$.  

\begin{prop}
\label{prop:distance}
$$d(x,y)=2h(x \vee y) -h(x)-h(y) = h(x)+h(y)-2h(x \wedge y) = h(x \vee y) - h(x \wedge y).
$$
\end{prop}

\begin{proof}
All but the first equality follow from the third part of Theorem ~\ref{thm:modular}.
The fact that $d(x,y) \leq 2h(x \vee y) -h(x)-h(y)$ follows by constructing a path
from $x$ to $x \vee y$ (by ascending the lattice) and then from $x \vee y$ to $y$
(descending the lattice). 

Next we show $d(x,y) \geq  h(x \vee y) - h(x \wedge y)$.  We prove this
by induction on $d(x,y)$.  The result is trivial when $d(x,y)=1$.
Suppose $d(x,y) >1$ and  let $x=x_0,\ldots,x_s=y$
be a shortest path from $x$ to $y$.  Then $x_1,\ldots,x_s$ must be a shortest path
from $x_1$ to $x_s$ and so by induction $s-1=h(x_1 \vee y) - h(x_1 \wedge y)$.
Assume $x_1$ covers $x$ (the case that $x$ covers $x_1$ is similar).
By Proposition \ref{prop:vee}, $h(x \vee y) - h(x \wedge y) \leq h(x_1 \vee y)-h(x_1 \wedge y)+1 \leq (s-1)+1=s$.
\end{proof}

In any path from $x=x_0,x_1,\ldots,x_s=y$, we label the $i$th step as an {\em up step}
if $x_{i-1}$ is covered by $x_i$ and a {\em down step} if $x_{i-1}$ covers $x_i$.
A path is an {\em up-down} path if all of the up steps precede all of the down steps
and is a {\em down-up} path if all of the down steps precede all of the up steps.

\begin{prop}
\label{lemma:up-down 1} 
For any up-down path from $x$ to $y$ of length $d(x,y)$ in $C_L$, the final vertex of the final up step
is  $x \vee y$ and for any down-up
path of length $d(x,y)$ the final vertex at the end of the final down step is $x \wedge y$.
\end{prop}

\begin{proof}
We prove the first statement, the second is proven analogously.  Let $P$ be an up-down
path of length $d(x,y)$ from $x$ to $y$ and let $z$ be the final vertex of the final up step.
We claim $z=x \vee y$.  Since $z \geq x$ and $z \geq y$, we have
$z\geq x \vee y$.  The path consists of $h(z)-h(x)$ up steps and $h(z)-h(y)$
down steps, and so has length $2h(z) -h(x)-h(y)$.  Since this equals
$d(x,y)$, Proposition~\ref{prop:distance} implies that $h(z)=h(x \vee y)$. 
 Since $z \geq x \vee y$ we must have $z = x \vee y$.
\end{proof}

\begin{lemma}
\label{lemma:up-down 2}  Suppose that $S$ is a diamond-closed subset of $C_L$.
Suppose that $S$
has   a path of length $d(x,y)$ from $x$ to $y$.
Then $S$ contains both an up-down path from $x$ to $y$ of
length $d(x,y)$  and a down-up path from $x$ to $y$ of length $d(x,y)$.
\end{lemma}

\begin{proof}
Assume $S$ contains a path  from $x$ to $y$
of length $d(x,y)$.  We prove that $S$  has an  up-down path  of the same
length (the argument for a down-up path is similar). 
For a path $P$, and for $1 \leq i < j \leq d(x,y)$ we say that $i,j$ is an {\em inversion}
if step $i$ is a down step and step $j$ is an up step.  Let $P=x_0,\ldots,x_{d(x,y)}$ be the induced path
of length $d(x,y)$ from $x$ to $y$ with the fewest number of inversions.  We claim that $P$ has 0 inversions,
and so is an up-down path.
Suppose for contradiction that $P$ has an inversion.  Then there is
an index $i$ such that step $i$ is down and step $i+1$ is up.
Thus $\dpair{x_{i-1}}{x_{i+1}} \rightarrow x_{i}$ is an in-V, and  Since $S$ is diamond-closed,
$S$ contains the out-V $x_{i-1} \wedge x_{i+1} \rightarrow \dpair{x_{i-1}}{x_{i+1}}$.  Thus
we can replace
$x_i$ by $x_{i-1} \vee x_{i+1}$ in the path to get a path in $S$ with
fewer  inversions to contradict
the choice of $P$.  Therefore $P$ has 0 inversions.
\end{proof}

\begin{lemma}
\label{lemma:short path}  Suppose $S$ is a connected diamond-closed subset of $C_L$.  
If $x,y \in V(S)$ then
there is a path of length $d(x,y)$ from $x$ to $y$.
\end{lemma}

\begin{proof}
Let $e(x,y)$ be the minimum length of an $x-y$-path in $S$; we need
$e(x,y) \leq d(x,y)$ which, by
Proposition ~\ref{prop:distance}, follows from
$e(x,y)\leq h(x \vee y) - h(x \wedge y)$.  We proceed by induction on $e(x,y)$.
If $e(x,y)=1$ then $x$ covers $y$ or $y$ covers $x$, and so $h(x \vee y)-h(x \wedge y)=1$.
Assume $e(x,y)>1$.  Let $N$ be the set of vertices  adjacent to $x$ on some $x-y$ path in $S$.
For all $z \in N$,   $e(z,y)=e(x,y)-1$, and by induction $e(z,y)=h(z \vee y)-h(z \wedge y)$. It suffices to show
there is a $w \in N$ such that:

\begin{equation}
\label{eqn:rank}
h(x \vee y)-h(x \wedge y) \geq  h(w \vee y) - h(w \wedge y) + 1.
\end{equation}

Select  $z \in N$. Assume $x$ covers $z$; the case $z$ covers $x$ is similar.  By Proposition ~\ref{prop:vee} either $x \wedge y = z \wedge y$ and
$x \vee y$ covers $z \vee y$ or $x \wedge y$ covers $z \wedge y$ and $x \vee y=z \vee y$.  If the former, then
(\ref{eqn:rank}) holds with $w=z$, so   assume the latter.   By lemmas ~\ref{lemma:up-down 2}
and ~\ref{lemma:up-down 1}, $S$ has  an up-down path $P=z,z_1,\ldots,y$  of length $d(z,y)$ whose up-segment ends with $z \vee y$.  
Since $z \vee y= x \vee y>z$, the up-segment 
is non-empty and so $z<z_1 \leq x \vee y$.   Let $z'=x \vee z_1$, we claim $z' \in N$:  Since
$\dpair{x}{z_1}\rightarrow z$ is in $S$ and $S$ is diamond-closed,
$S$ contains $z' \rightarrow \dpair{x}{z_1}$.  Replacing $z$ by $z'$ in the $x-y$ path $x,P$ shows
$z' \in N$.  Since 
$z'$ covers $x$ and $z' \vee y=x \vee z_1 \vee y  \leq x \vee y$,
Proposition ~\ref{prop:vee} implies $z' \wedge y$ covers $x \wedge y$ and (\ref{eqn:rank}) holds with $w=z'$.
\end{proof}

We are now ready to prove the implication $(1) \Rightarrow (2)$ of the theorem.  Let $S$ be
diamond-closed; as noted
earlier, we may assume $S$ is connected. Let $K$ be the set of elements of $L$ that belong to some
arc of $S$.  We show that $K$ is a cover preserving sublattice and that $S=C_L[K]$.

Suppose $x,y \in K$.
Lemma ~\ref{lemma:short path} implies that $S$ contains a path
of length $d(x,y)$ from $x$ to $y$, and therefore
Lemma ~\ref{lemma:up-down 2} implies that $S$ contains  both
an up-down path and a down-up path of length $d(x,y)$
and so by Proposition ~\ref{lemma:up-down 1}, $x \vee y$ and $x \wedge y$
belong to $K$, and so $K$ is a sublattice.

Next we show that $K$ is cover-preserving. 
Since $S \subseteq C_L[K]$, by Proposition ~\ref{prop:cover preserving} it suffices to consider
a pair $x,y \in K$ with $y>x$ and show
that $S$ contains a directed path from $y$ to $x$.
By Lemma ~\ref{lemma:up-down 2}, there is  an up-down path in $S$ from $y$ to $x$
of length $d(x,y)$. By Lemma ~\ref{lemma:up-down 1} the up segment of this path ends with $x \vee y=y$, which
means the path has no up steps and is therefore directed.  

Finally we show $S=C_L[K]$. By definition of $K$, $S \subseteq C_L[K]$.  To show the reverse implication,
suppose $y \rightarrow x \in C_L[K]$. By the previous paragraph, there is a directed path from $y$ to $x$ in $S$
which must consist of the single arc $y \rightarrow x$ and so $y \rightarrow x \in S$.

This completes the 
proof of Theorem ~\ref{thm:diamond-closed}.
\end{proof}

We note the following:

\begin{corollary}
\label{cor:1 to 1}
For any finite modular lattice $L$:
\begin{enumerate}
\item The mapping $K \longrightarrow C_K$ is a one-to-one correspondence between cover-preserving sublattices of $L$
and connected diamond-closed subsets of $C_L$.
\item
The mapping $\mfK \longrightarrow C_{\mfK}$ is a one-to-one correspondence between CS-packings of $L$
and diamond-closed subsets of $L$. 
\end{enumerate}
\end{corollary}

\subsection{Constructing the diamond-closure in finite modular lattices}
\label{subsec:ML diamond-closure}

The results of the previous subsection yield an explicit description of the closure of
a set $A$ of arcs.  For the case that $A$ is connected we have:

\begin{corollary}
\label{cor:closure 1}
Let $L$ be a finite modular lattice and let $A$ be a connected subset of  $C_L$. Let $L(A)$ be the
set of vertices of $L$ spanned by $A$ (i.e., that belong to some arc of $A$) and let $K=\langle L(A) \rangle_L$
be the sublattice of $L$ generated by $L(A)$.  Then  $K$ is cover-preserving and $\bigdiamond(A)$ is equal to the
set  $C_L[K]$ of arcs induced on the sublattice $K$.
\end{corollary}

\begin{proof}
$\bigdiamond(A)$ is connected since each  arc-component of $\bigdiamond(A)$ is diamond-closed, and so the
arc-component containing $A$ is a closed subset of $\bigdiamond(A)$ that contains $A$ and so must equal
$\bigdiamond(A)$.  By Corollary ~\ref{cor:1 to 1}, $\bigdiamond(A)=C_L[J]$ for some cover-preserving sublattice $J$ of $L$, and we claim
$J=K$.
$A \subseteq C_L[J]$ implies $V(A) \subseteq J$ and hence $K = \langle V(A) \rangle_L \subseteq J$.
 $C_L[K]$ is diamond-closed and contains $A$, so $C_L[J] = \bigdiamond(A) \subseteq C_L[K]$ and hence
$J \subseteq K$.  
\end{proof}

For general $A$, Corollary ~\ref{cor:1 to 1} implies that there is a unique CS packing $\mfK(A)$ such that
$\bigdiamond(A) = C_{\mfK(A)}$.
The following procedure can be used to determine $\mfK(A)$:

\begin{enumerate}
\item Let $\mathfrak{Y}$ be the set of vertex components of $A$.
\item Let $\mathfrak{K}=\{\langle Y \rangle_L:Y \in \mathfrak{Y}\}$.
\item  While
 $\mathfrak{K}$ contains two distinct sublattices $K,K'$ with nonempty intersection,
remove $K,K'$ from $\mathfrak{K}$ and add $\langle K \cup K' \rangle_L$.
\end{enumerate}

\begin{thm}
\label{thm:MLDC}
For any finite modular lattice $L$ and $A \subseteq C_L$, the above procedure terminates 
with $\mathfrak{K}=\mfK(A)$.
\end{thm}

\begin{proof}
Since $\mathfrak{K}$ is initially finite, and shrinks in size by one during each iteration of the while loop, this procedure terminates.
Let $\mathfrak{K}^*$ be the final value of $\mathfrak{K}$ and let $\mathfrak{K}^0$ be the initial value of $\mathfrak{K}$. 

We claim 
$\bigdiamond(C_L[\mathfrak{K}])=\bigdiamond(A)$ holds throughout the procedure.   It is clear from the procedure that $A \subseteq C_L[\mathfrak{K}^0]$ and that
$C_L[\mathfrak{K}]$ can only grow, and therefore $A \subseteq C_L[\mathfrak{K}]$ and $\bigdiamond(A) \subseteq
\bigdiamond(C_L[\mathfrak{K}])$.  The claim will then follow 
by showing $C_L[\mathfrak{K}] \subseteq \bigdiamond(A)$ after every iteration.  We prove this by induction on the number
of iterations.  For $\mathfrak{K}^0$, let $\mathfrak{A}$ denote the
set of arc-components of $A$. For each $B \in \mathfrak{A}$, let $Y(B)$ be the associated vertex component, and let $K(B)=\langle Y(B) \rangle_L$.
By Corollary ~\ref{cor:closure 1}, $C_L[K(B)] = \bigdiamond(B) \subseteq \bigdiamond(A)$ and so 
$C_L[\mathfrak{K}^0]=\bigcup_{B \in \mathfrak{A}}C_L[K(B)] \subseteq \bigdiamond(A)$.

Now for a given iteration, assume by induction that  $C_L[\mathfrak{K}] \subseteq \bigdiamond(A)$ holds prior
to the iteration.   Let $K,K'$ be members of $\mathfrak{K}$
with nonempty intersection.  Then $C_L[\langle K \cup K'\rangle_L]=\bigdiamond(C_L[K \cup K']) \subseteq \bigdiamond(C_L[\mathfrak{K}]) \subseteq
\bigdiamond(A)$ and so $C_L[\mathfrak{K}-\{K,K'\} \cup \langle K \cup K' \rangle_L] \subseteq \bigdiamond(A)$.

So $\bigdiamond(C_L[\mathfrak{K}^*] )= \bigdiamond(A)$ and since $\mathfrak{K}^*$ is a CS-packing, Corollary ~\ref{cor:1 to 1} implies
$\mfK(A)=\mathfrak{K}^*$.
\end{proof}

\section{Diamond-closure in  distributive lattices}
\label{sec:connected}

Theorem ~\ref{thm:MLDC} describes the
diamond-closure operation for a finite modular lattice.  In this section
we give a simplified description in the case that $L$ is a finite distributive lattice.

We first use the Birkhoff representation theorem for distributive lattices to give a convenient way to describe the sublattices and
the cover-preserving sublattices of a distributive lattice. 

\subsection{Representing distributive lattices and sublattices}

\iffalse
Recall that an alignment $\mathcal{A}$ on $X$ is a family of subsets that contains $X$ and is closed under intersection.
Every alignment is a lattice with meet operation given by $\cap$.  The join operation given for $\mathcal{W} \subseteq \mathcal{A}$
by $\bigvee \mathcal{W}$ defined to be the intersection of all members of $\mathcal{A}$ that contain all elements
of $\mathcal{W}$.  Thus the join of $\mathcal{W}$ contains $\bigcup \mathcal{W}$, but is not necessarily equal to it.
\fi

Let $(Q,\leq)$ be a set with a quasi-order (i.e. a transitive and reflexive relation)
Such a relation splits $Q$ into equivalence classes where 
$x$ and $y$ are equivalent provided that $x \leq y$ and $y \leq x$ and $\leq$ induces a partial
order on the equivalence classes.   A partial order is a quasi-order
where all equivalence classes have size 1. 

We adopt the convention that every
quasi-ordered set is equipped with distinguished 
elements $\hat{0}_Q$ and $\hat{1}_Q$  such that $\hat{0}_Q \leq  x \leq \hat{1}_Q$ for all $x \in Q$ and $\hat{0}_Q<\hat{1}_Q$.
We refer to such a quasi-order as {\em pointed}.
The equivalence classes of $\hat{0}_Q$ and $\hat{1}_Q$ may have size larger than 1,
but these classes must be distinct.    A {\em downset} of $(Q,\leq)$ is a subset $D$ satisfying (1) if $x \leq y \in Q$ and $y \in D$
then $x \in D$, (2) $\hat{0}_Q \in D$ and (3) $\hat{1}_Q \not\in D$.    

The requirement that $Q$ be  pointed  and that a downset satisfy conditions (2) and (3)
are non-standard but useful for stating our results, especially the description of connected diamond-closed
subsets given by Theorem ~\ref{thm:DL connected DC}. 

An {\em extension} of a pointed quasi-order $(Q,\leq)$ is a pointed  quasi-order
$(Q,\leq^*)$  that is at least as strong as $\leq$,
i.e. $x \leq y$ implies $x \leq^* y$. 

The set $\mathcal{D}(Q,\leq)$ 
of all downsets of $(Q,\leq)$ is an alignment on $Q$. When this alignment is viewed as a lattice under set inclusion, $\wedge$
corresonds to set-intersection (as for any alignment) and
$\vee$  coincides to set union since $\mathcal{D}(Q,\leq)$ is closed under union.
For example, if $(Q,\leq)$ is the trivial ordering containing only relations $x \leq x$, then $\mathcal{D}(Q,\leq)$
is just the lattice $2^Q$ of all subsets of $Q$.    In general,
$\mathcal{D}(Q,\leq)$ is always a sublattice of the lattice $2^Q$ of all subsets of $Q$, and is thus distributive.
These are the only finite distributive lattices:

\begin{thm}
\label{thm:birkhoff}
(Birkhoff \cite{Birk})
For every finite distributive lattice $L$, there is a pointed partially ordered set $(P,\leq)$, unique up to isomorphism, 
such that $L$ is isomorphic to $\mathcal{D}(P,\leq)$.       
\end{thm}

Note that if $(Q,\leq)$ is a quasi-order then 
$\DD(Q,\leq)$
is isomorphic to $\DD(P,\leq)$ where 
$(P,\leq)$ is the pointed partial order on equivalence classes of $Q$
mentioned earlier. 

In what follows we consider an arbitrary finite distributed lattice represented as $(P,\leq)$ 
for some pointed partial order $(P,\leq)$. As
the members of $\DD(P,\leq)$ are sets, we   denote this lattice by $\LL$ rather than $L$,
and denote subsets of $\LL$ by calligraphic letters, and members of $\LL$ by upper case letters.  

As we now describe, this representation 
provides an easy way to describe  sublattices of $\LL$, cover-preserving sublattices of $\LL$, the sublattice generated
by an arbitrary subset of $\LL$, and the diamond closure of subsets of arcs of the closure graph.

Theorem ~\ref{thm:DL generated} establishes a natural correspondence between sublattices of $\mathcal{D}(P,\leq)$ and
extensions of $(P,\leq)$.  In preparation we need some notation.

For $\YY \subseteq \DD(P,\leq)$,  the partial order $\leq^{\YY}$ on $P$ is defined by
$i \leq^{\YY} j$
provided that 
every member of $\YY$ that contains $j$ also contains $i$.    If $i \leq j$, then
we necessarily have $i  \leq ^{\YY}j$ (since every downset of $(P,\leq)$ that contains $j$
also contains $i$) and so 
 $(P,\leq^{\YY})$ is an extension of $(P,\leq)$.

The following  is a natural extension of Lemma 4.3 from \cite{siggers2014} with some differences in notation.

\begin{thm}
\label{thm:DL generated}
Let $(P,\leq)$ be a pointed partially ordered set, and let $\LL=\DD(P,\leq)$.  There is a one-to-one correspondence between
the sublattices of $\LL$ and the extensions $(P,\leq^*)$ of $(P,\leq)$ given by the following inverse maps:

\begin{itemize}
\item Extension $(P,\leq^*)$ maps to sublattice $\DD(P,\leq^*)$ of $\LL$.
\item Sublattice $\MM$ maps to extension  $(P,\leq^{\MM})$ 
\end{itemize}

Furthermore, for an extension $(P,\leq^*)$, $\DD(P,\leq^*)$ 
is a cover-preserving sublattice of $\LL$  if and only if every equivalence classes of $(P,\leq^*)$
other than the equivalence classes $E_0$ and $E_1$ of $\hat{0}$ and $\hat{1}$
have size exactly one.  
\end{thm}

\begin{proof}
First we show that the map $(P,\leq^*)$ to $\DD(P,\leq^*)$ is a one-to-one map
from the set of extensions of $(P,\leq)$ to the set of sublattices of $\LL$.
For any extension $(P,\leq^*)$  of  $(P,\leq)$,  
$\mathcal{D}(P,\leq^*)$ is a sublattice of $\LL$ that is closed under intersection and union.
The map is one-to-one: If $(P,\leq^*)$ and $(P,\leq^\#)$ are distinct quasi-orders, without loss of
generality suppose $i \leq^* j$ and $i$ is not $\leq^\# j$.  $\DD(P,\leq^{\#})$ includes the
set $\{h \in P:h \leq^{\#} j\}$ and since this set contains $j$ and not $i$, it is not a member of
$\DD(P,\leq^*)$.  Hence $\DD(P,\leq^*) \neq \DD(P,\leq^{\#})$.  

Next we show that if 
$\MM$ is any sublattice
then $\DD(P,\leq^{\MM})=\MM$.   If $D \in \MM$ then $D \in \DD(P,\leq^{\MM})$
since for any $i,j \in \LL$, if $j \in D$ and $i \leq^{\MM} j$ then $i \in D$ (since every set belonging to $\MM$
that contains $j$ also contains $i$ by the definition of $\leq^{\MM}$).  Thus $D \in \DD(P,\leq^{\MM})$ and so 
$\MM \subseteq \DD(P,\leq^{\MM})$.
Now suppose $D \in \DD(P,\leq^{\MM})$.   For each $j \in D$,  let $D_j=\{i \in P:i \leq^{\MM} j\}$. Then $D_j \subseteq D$
since $D \in \DD(P,\leq^{\MM})$ and so $D=\bigcup_{j \in D}D_j$.  But also $D_j$ is equal to the intersection of all members of $\MM$
that contain $j$. Since $\MM$ is a sublattice, each $D_j \in \MM$ and also $D=\bigcup_{j \in D}D_j \in \MM$.
Thus $\DD(P,\leq^{\MM}) \subseteq \MM$ and so $\DD(P,\leq^{\MM})=\MM$.    

This implies that
the map sending $(P,\leq^*)$ to $\DD(P,\leq^*)$ is a bijection from the set of extensions $(P,\leq^*)$ of $(P,\leq)$
to the set of sublattices of $\LL$, and that the inverse map sends the sublattice $\MM$ of $\LL$ to the
extension $(P,\leq^*)$.

Suppose now $(P,\leq^*)$ is an extension of $(P,\leq)$.  We want to show $\DD(P,\leq^*)$ is cover-preserving if and only if the classes of $(P,\leq^*)$ other
than $E_0$ and $E_1$ have size 1.  

We first note the following easy fact.
\begin{proposition}
\label{prop:cover}
Suppose $(P,\leq^*)$ is an extension of $(P,\leq)$ and $D \in \DD(P,\leq^*)$ and $C \subset D$.  Then $C \in \DD(P,\leq^*)$ and
$D$ covers $C$ in $\DD(P,\leq^*)$ if and only if $D-C$ is an $\leq^*$-equivalence class  that is $\leq^*$-maximal among equivalence
classes inside $D$.
\end{proposition}

\begin{proof}
Suppose $D-C$ is $\leq^*$-equivalence class  that is $\leq^*$-maximal among equivalence classes inside $D$.  We claim that
$C \in \DD(P,\leq^*)$.  Let  $j \in C \subseteq D$ and $i<^*j$.  Then $i \in D$ since $D \in \DD(P,\leq^*)$.  Now $i \not\in D-C$ since
$D-C$ is maximal among equivalence classes inside $D$ and is therefore not $\leq^* j$.   Furthermore $D$ covers $C$ in $\DD(P,\leq^*)$ since  for $Z$  satisfying $C \subset Z \subset D$,
$Z$ is not a union of $\leq^*$-equivalence classes and so is not in $\DD(P,\leq^*)$.

Conversely, suppose $C \in \DD(P,\leq^*)$ and $D$ covers $C$ in $\DD(P,\leq^*)$.  Then $D-C$ is a union of $\leq^*$-equivalence classes. Let $E \subseteq D-C$ be an equivalence class that is $\leq^*$-maximal among classes contained in $D-C$.  
Since $C$ is a downset, no element of $C$ is above any element of $E$ and so $E$ is also maximal among classes in $D$.  Therefore $D$ covers $D-E \geq^*C$ in $\DD(P,\leq^*)$ which
implies $D-E=C$ since $D$ covers $C$ in $\DD(P,\leq^*)$.
\end{proof}

Now suppose $\DD(P,\leq^*)$ is a cover-preserving  sublattice of $\DD(P,\leq)$ and suppose $E$ is an $\leq^*$-equivalence class 
other than
$E_0$ and $E_1$. Select an arbitrary  $j \in E$  and let $D=\{i \in P: i \leq^* j\}$.  Then $D \in \DD(P,\leq^*)$ and $E \subseteq D$
is a $\leq^*$-maximal equivalence class in $D$.  By Proposition ~\ref{prop:cover}, $D$ covers $D-E$ in $\DD(P,\leq^*)$ and since 
$\DD(P,\leq^*)$
is cover-preserving, $D$ covers $D-E$ in $\DD(P,\leq)$.  By Proposition ~\ref{prop:cover} applied to $\DD(P,\leq)$, $E$
is a $\leq$-equivalence class, and therefore has size 1 since $\leq$ is a partial order.   

Conversely, suppose all $\leq^*$-equivalence classes other than $E_0$ and $E_1$ have size 1.  Let $C,D \in \DD(P,\leq^*)$
with $D$ covering $C$.  By Proposition ~\ref{prop:cover}, $D-C$ is a $\leq^*$-equivalence class.  It is not $E_1$ (since $\hat{1} \not\in D$), and it is not $E_0$ (since $\hat{0} \in C$),  and therefore it has size 1. Since $C \subset D \in \DD(P,\leq)$ and $|D-C|=1$ we also
have $D$ covers $C$ in $\DD(P,\leq)$.  Therefore $\DD(P,\leq^*)$ is cover-preserving.

\end{proof}

\begin{corollary}
\label{cor:DL connected closure}
Let $(P,\leq)$ be a finite partial order and suppose $\SSS$ is a subset of $\DD(P,\leq)$.   
\begin{enumerate}
\item The
sublattice generated by $\SSS$ is equal to $\DD(P,\leq^{\SSS})$.  
\item  If $\SSS$ is connected then $\DD(P,\leq^{\SSS})$ is cover-preserving.
\end{enumerate}
\end{corollary}

\begin{proof}
For the first part,
let $\KK$ be the sublattice generated by $\SSS$; we claim $\KK = \DD(P,\leq^{\SSS})$.
Since $\SSS$ is a subset of the sublattice  $\DD(P,\leq^{\SSS})$, then $\KK \subseteq \DD(P,\leq^{\SSS}~)$.
Also, Theorem ~\ref{thm:DL generated} implies that
$\KK=\DD(P,\leq^{\KK})$, and therefore 
$\SSS \subseteq \KK$ implies $\DD(P,\leq^{\SSS}) \subseteq \DD(P,\leq^{\KK})=\KK$.

The second part is an immediate consequence of Corollary ~\ref{cor:closure 1}.
\end{proof}

We can now describe the
diamond-closure of a connected subset $A$ of arcs in the cover graph of a distributive lattice.

\begin{thm}
\label{thm:DL connected DC}
Let $\LL=\mathcal{D}(P,\leq)$ and let $A$ be a connected subset of the
cover graph $C_{\LL}$.  Let $\LL(A)$ be the set of members of $\LL$ 
belonging to at least one arc of $A$.
Then $\bigdiamond(A)$ is equal to the cover graph 
of the sublattice $\mathcal{D}(P,\leq^{\LL(A)})$.  In particular,
$\bigdiamond(A)$ is the entire cover graph of $\LL$ if and only if for every $i,j \in P$
for which $j$ is not $\leq i$,
there is a  $D \in \LL(A)$ that contains $i$ and not $j$.
\end{thm}

\begin{proof}
By Theorem ~\ref{thm:MLDC}, $\bigdiamond(A)$ is equal to the cover graph $C_{\KK}$,
where $\KK$ is the sublattice $\langle \LL(A) \rangle_{\LL}$. By Corollary ~\ref{cor:DL connected closure}
 $\KK=\mathcal{D}(P,\leq^{\LL(A)})$.

Now $\bigdiamond(A)=C_{\LL}$ if and only if the $\leq^{\LL(A)}$ order coincides with $\leq$
which means that for all $i,j$, if $j \not \leq i$, there is a  $D \in \LL(A)$ that contains $i$
and not $j$.
\end{proof}

\begin{example} Suppose $\LL$ is the lattice of all subsets of $[n]=\{1,2,\dots,n\}$. 
Each of the following two subsets of arcs have diamond-closure equal to the entire cover graph of $\LL$:
\begin{itemize}
\item
The set of arcs $\{\{i\} \rightarrow \emptyset:i \in [n]\}$.
\item 
The set of arcs $\{\{1,\ldots,i+1\} \rightarrow \{1,\ldots,i\}:1 \leq i \leq n-1\} \cup \{\{n-i-1,\ldots,n\} \rightarrow \{n-i,\ldots,n\}:0 \leq i \leq n-2\}$
\end{itemize}
\end{example}

For disconnected subsets, we can specialize the general procedure for
diamond-closure for modular lattices in Section ~\ref{subsec:ML diamond-closure}  
to distributive lattices.  As above, let $\LL=\DD(P,\leq)$
be a distributive lattice.  
The two main tasks required to carry out this procedure
are (1) Determining the sublattice generated by a connected subset of elements, and
(2) Determining whether two sublattices have nonempty intersection. 
By Theorem ~\ref{thm:DL generated}, the sublattice generated by $\YY$ is $\mathcal{D}(P,\leq^{\YY})$, which
implements (1).  To carry out (2), we make the following definitions: If $(P,\leq^1)$ and $(P,\leq^2)$
are quasi-orders on the same set $P$ then $(P,\leq^1 \bar{\cup} \leq^2)$ is the transitive relation
obtained by taking the transitive closure of the union of the relations.   We say
that $(P,\leq^1)$ and $(P,\leq^2)$ are {\em compatible} if
in $(P,\leq^1 \bar{\cup} \leq^2)$, $\hat{0}_P$ and $\hat{1}_P$ are in different
equivalence classes.  

\begin{prop}
\label{prop:intersection}
Let $\LL=\DD(P,\leq)$ be a distributive lattice and let $(P,\leq^1)$
and $(P,\leq^2)$ be extensions.   Then 
$\DD(P,\leq^1) \cap \DD(P,\leq^2)$  is nonempty if and only
$(P,\leq^1)$ and $(P,\leq^2)$ are compatible.
\end{prop}

\begin{proof} 
Let $\leq^3$ denote the transitive closure of $\leq^1 \cup \leq^2$.
Let $D=\{j \in P:j \leq^3 \hat{0}_P\}$.  If $\leq^1$ and $\leq^2$ are
compatible then $\hat{1}_P \not\in D$, and so
then  $D \in \mathcal{D}(P,\leq^1) \cap \mathcal{D}(P,\leq^2)$. 
Conversely, suppose 
$C \in \mathcal{D}(P,\leq^1) \cap \mathcal{D}(P,\leq^2)$. Then
for all $i\in P-C$ and $j \in C$ we have neither $i \leq^1 j$ nor $i \leq ^2 j$,
and thus the same property holds for $\leq^3$.  Therefore $\hat{1}_P$ and
$\hat{0}_P$ are in different equivalence classes with respect to $\leq^3$.
\end{proof}

So if $(P,\leq^1)$ and $(P,\leq^2)$ are compatible, then
$\mathcal{D}(P,\leq^1) \cup \mathcal{D}(P,\leq^2)$ is connected.
The lattice spanned by this union corresponds to $\mathcal{D}(P,\leq^{\YY})$
where $\YY=\DD(P,\leq^1~) \cup \DD(P,\leq^2)$.  It is easy to verify that 
$i \leq^{\YY} j$ if and only if $i \leq^1 j$ and $i \leq^2 j$,
so that $\leq^{\YY} = \leq^1 \cap \leq^2$. 
Using this, we obtain the following procedure for obtaining the diamond-closure
of an arbitrary subset $A$ of the arcs of $\mathcal{D}(P,\leq)$.

\begin{enumerate}
\item Let $\{A_1,\ldots,A_k\}$ be the connected components of $A$.
\item For $i \in \{1,\ldots,k\}$ let $\YY_i=\LL(A_i)$, and let $(P,\leq^i)$
be the extension $(P,\leq^{\YY_i})$.   Let $\mathcal{P}$ be the
set consisting of all of the quasi-orderings $\{\leq^i: i \in \{1,\ldots,k\}\}$.
\item  While $\mathcal{P}$ contains a pair $\leq^*$ and $\leq^{\circ}$ $\mathcal{P}$ of compatible
orders replace them in $\mathcal{P}$ by
$\leq^* \cap \leq^{\circ}$.
\end{enumerate}

\begin{thm}
\label{thm:DLDC}
For any finite distibutive lattice $\LL=\DD(P,\leq)$ and $A \subseteq C_{\LL}$, the above procedure produces a collection $\mathcal{P}$ of  quasi-orders on $P$, and
$\bigdiamond(A)$ is equal to  union of $C_{\DD(P,\leq^*)}$ over quasi-orders
in $\mathcal{P}$.
In particular $\bigdiamond(A)=C_{\LL}$ if and only if the procedure
ends with $\mathcal{P}$ having $\leq$ as its only member.
\end{thm}

\section*{Acknowledgements.} We thank an anonymous referee for carefully reading the paper and for identifying technical issues that have been corrected for this final version.

\end{document}